\documentclass[11pt]{amsart}
\usepackage{graphicx,amssymb,amsmath,amsthm}
\usepackage{enumerate}
\usepackage{dsfont}
\usepackage[colorlinks, citecolor=red]{hyperref}
\usepackage{comment,cite,color}
\usepackage{cite,color}
\usepackage{mathrsfs}
\usepackage{epsfig}
\usepackage{lscape}
\usepackage{subfigure}
\usepackage{epstopdf}
\usepackage{caption}
\usepackage{algorithm}

\usepackage{tgtermes}
\usepackage{lipsum}

\usepackage{bm}

\usepackage{algpseudocode}

\newcommand{\xkh}[1]{\left(#1\right)}
\newcommand{\dkh}[1]{\left\{#1\right\}}
\newcommand{\zkh}[1]{\left[#1\right]}

\newcommand{\nj}[1]{\langle {#1} \rangle}

\newcommand{\norm}[1]{\|{#1}\|_2}

\newcommand{\abs}[1]{\left\lvert#1\right\rvert}

\newcommand{\E}{{\mathbb E}}

\newcommand{\PP}{{\mathbb P}}

\newcommand{\1}{{\mathds 1}}

\newcommand{\R}{{\mathbb R}}

\newcommand{\C}{{\mathbb C}}
\newcommand{\CS}{{\mathbb {S}}}

\newcommand{\vx}{{\bm x}}

\newcommand{\vy}{{ \bm{ y}}}

\newcommand{\vv}{{\bm v}}
\newcommand{\vz}{{\bm z}}

\newcommand{\ve}{{\bm e}}

\newcommand{\va}{{\bm a}}
\newcommand{\vh}{{\bm h}}

\newcommand{\cN}{{\mathcal N}}

\newcommand{\dist}{{\rm dist}}

\renewcommand{\omega}{\eta}

\newcommand{\RNum}[1]{\uppercase\expandafter{\romannumeral #1\relax}}

\newtheorem{definition}{Definition}[section]

\newtheorem{theorem}[definition]{Theorem}
\newtheorem{lemma}[definition]{Lemma}

\newtheorem{remark}[definition]{Remark}

\date{}

\begin{document}
\baselineskip 19pt
\bibliographystyle{plain}
\title[Linear convergence of randomized Kaczmarz method]{Linear convergence of randomized Kaczmarz method for solving complex-valued phaseless equations}
\author{Meng Huang}
\address{Department of Mathematics, The Hong Kong University of Science and Technology,
Clear Water Bay, Kowloon, Hong Kong, China} \email{menghuang@ust.hk}

\author{Yang Wang}
\thanks{Y. Wang was supported in part by the Hong Kong Research Grant Council grants 16306415 and 16308518.}
\address{Department of Mathematics, The Hong Kong University of Science and Technology,
Clear Water Bay, Kowloon, Hong Kong, China} \email{yangwang@ust.hk}

\maketitle

\begin{abstract}
A randomized Kaczmarz method was recently proposed for phase retrieval, which has been shown numerically to exhibit empirical performance over other state-of-the-art phase retrieval algorithms both in terms of the sampling complexity and in terms of computation time. While the rate of convergence has been studied well in the real case where the signals and measurement vectors are all real-valued, there is no guarantee for the convergence in the complex case. In fact,  the linear convergence of the randomized Kaczmarz method for phase retrieval in the complex setting is left as a conjecture by Tan and Vershynin.
In this paper, we provide the first theoretical guarantees for it.
We show that for random measurements $\va_j \in \C^n, j=1,\ldots,m $ which are drawn independently and uniformly from the complex unit sphere, or equivalent are independent complex Gaussian random vectors,
when $m\ge Cn$ for some universal positive constant $C$, the randomized Kaczmarz scheme with a good initialization converges linearly to the target solution (up to a global phase) in expectation with high probability. This gives a positive answer to that conjecture.

\end{abstract}

\section{Introduction}
\subsection{Problem setup}
Let $\vx \in \C^n$ (or $\R^n$) be an arbitrary unknown vector. We consider the problem of recovering $\vx$ from the phaseless equations:
\begin{equation} \label{eq:problesetup}
b_j=\abs{\nj{\va_j,\vx}}, \quad j=1,\ldots,m,
\end{equation}
where $\va_j \in \C^n$ (or $\R^n$) are known sampling vectors and $b_j \in \R$ are observed measurements.
This problem, termed as {\em phase retrieval},  has been a topic of study from 1980s due to its wide range of practical applications in fields of physical sciences and engineering, such as X-ray crystallography \cite{harrison1993phase,millane1990phase}, diffraction imaging \cite{shechtman2015phase,chai2010array}, microscopy \cite{miao2008extending}, astronomy \cite{fienup1987phase}, optics and acoustics \cite{walther1963question, balan2010signal,balan2006signal} etc, where the  detector can record only the diffracted intensity while losing the phase information. 
Despite its simple mathematical form, it has been shown that to reconstruct a finite-dimensional discrete signal from its Fourier transform magnitudes is generally NP-complete \cite{Sahinoglou}.
Another special case of solving this phaseless equations is the well-known {\em stone problem} in combinatorial optimization, which is also NP-complete \cite{ben2001lectures}.

To solve \eqref{eq:problesetup}, we employ the randomized Kaczmarz method where the update rule is given by 
\begin{equation} \label{eq:updateformula}
\vz_{k+1}=\vz_k-\xkh{1-\frac{b_{i_k}}{|\va_{i_k}^*\vz_k|}} \frac{\va_{i_k}^*\vz_k}{\norm{\va_{i_k}}^2} \va_{i_k},
\end{equation}
where $i_k$ is chosen randomly from the $\dkh{1,\ldots,m}$ with probability proportional to $\norm{\va_{i_k}}^2$ at the $k$-th iteration.  
Actually, the update rule above is a natural adaption of the classical randomized Kaczmarz method \cite{karczmarz1937angenaherte} for solving linear equations.
The idea behind the scheme is simple. When the iteration is close enough to the signal vector $\vx$,  the phase information can be approximated by that of the current estimate. 
Thus,  in each iteration, we first select a measurement vector $\va_{i_k}$ randomly, and then project the current estimate $\vz_k$ onto the hyperplane 
\[
\dkh{\vz \in \C^n: \nj{\va_{i_k},\vz}=b_{i_k} \cdot \frac{\va_{i_k}^* \vz_k}{|\va_{i_k}^* \vz_k|}}.
\]
That gives the scheme \eqref{eq:updateformula}.

We are interested in the following questions:

 {\em Does the randomized Kaczmarz scheme \eqref{eq:updateformula} converge to the target solution $\vx$ (up to a global phase)  in the complex setting? Can we establish the rate of convergence?}


\subsection{Motivation}
The randomized Kaczmarz method for solving phase retrieval problem was proposed by Wei \cite{Wei2015} in 2015. It has been demonstrated in \cite{Wei2015} using numerical experiments that 
the randomized Kaczmarz method exhibits empirical performance over other state-of-the-art phase retrieval algorithms both in terms of the sampling complexity and in terms of computation time, when 
the measurements are real or complex Gaussian random vectors, or when they follow the coded diffraction pattern (CDP) model.
However, no adequate theoretical guarantee for the convergence  was established in \cite{Wei2015}. 
To bridge the gap, for the real Gaussian measurement vectors, Li et al.  \cite{ligen2015} establish an asymptotic convergence of the randomized Kaczmarz method for
phase retrieval, but it requires infinite number of samples, which is unrealistic in practical. 
Lately, Tan and Vershynin \cite{tan2019phase} use the chain argument coupled with bounds on Vapnik-Chervonenkis (VC) dimension and metric entropy,  and then prove theoretically that the randomized Kaczmarz method for phase retrieval is linealy convergent with $O(n)$ Gaussian random measurements, where $n$ is the dimension of the signal.  
A result  almost same to that of \cite{tan2019phase}   is also obtained independently by Jeong and Gunturk \cite{jeong2017convergence} using the tools of hyperplane tessellation and ``drift analysis''. 
Another similar conditional error contractivity result is also established by Zhang et al. \cite{RWF}, which is called {\em incremental reshaped Wirtinger flow}.

We shall emphasize that all results concerning the convergence of the randomized Kaczmarz method for phase retrieval are for the real case where the signals and measurement vectors are all real-valued.
Since the phase can only be $+1$ or $-1$ in the real case, then the measurement vectors can be divided into  ``good measurements'' with correct phase  and ``bad measurements'' with incorrect phase. 
When the initial point is close enough to the true solution, the total influence of ``bad measurements''  can be well controlled.
However, this is not true for the complex measurements because $\vx e^{i \theta}$  is continuous with respect to $\theta\in [0, 2\pi)$.
For this reasons, the proofs for the real case can not be generalized to the complex setting easily.
As stated in  \cite[Section 7.2]{tan2019phase}, the linear convergence of the randomized Kaczmarz method for phase retrieval in the complex setting is left as a {\em CONJECTURE}.
We shall point out that the convergence of randomized Kaczmarz method for phase retrieval in complex setting is of more  practical  interest.

In this paper, we aim to prove this conjecture by introducing a deterministic condition on measurement vectors called ``Restricted Strong Convexity''  and then showing that the random measurements drawn independently 
and uniformly from the complex-valued sphere, or equivalently for the complex Gaussian random vectors,  satisfy this condition with high probability, as long as the measurement number $m\ge O(n)$.

\subsection{Related Work}
\subsubsection{Phase retrieval}
The phase retrieval problem, which aims to recover $\vx$ from phaseless equations \eqref{eq:problesetup}, has received intensive investigations recently.
Note that if  $\vz$ is a solution to \eqref{eq:problesetup} then $\vz e^{i\theta}$ is also the solution of this problem for any $\theta\in \R$. Therefore the recovery of the solution $\vx$ is up to a global phase.
It has been shown theoretically that $m\ge 4n-4$  generic measurements  suffice to recover $\vx$  for the complex case \cite{conca2015algebraic,wangxu} and $m\ge 2n-1$ are sufficient for the real case \cite{balan2006signal}.

Many algorithms with provable performance guarantees have been designed to solve the phase retrieval problem. One line of research relies on a “matrix-lifting” technique, 
which lifts the phase retrieval problem into a low rank matrix recovery problem, 
and then a nuclear norm minimization is adopted as a convex surrogate of the rank constraint. Such methods include PhaseLift \cite{phaselift,Phaseliftn}, PhaseCut \cite{Waldspurger2015}  etc.
While this convex methods have a substantial advance in theory, they tend to be computationally inefficient for large scale problems.
Another line of research seeks to optimize a non-convex loss function in the natural parameter space,  which achieves significantly improved computational performance.
The first non-convex algorithm with theoretical guarantees was given by Netrapalli et al  who proved that the AltMinPhase \cite{AltMin}  algorithm, based on a technique known as {\em spectral initialization},
converges linearly to the true solution up to a global phase with $O(n \log^3 n)$ resampling Gaussian random measurements. 
This work led to further several other non-convex algorithms based on spectral initialization \cite{bendory2017non, WF, TWF, waldspurger2018phase,huang2020solving}. 
Specifically, Cand\`es et al developed the Wirtinger Flow (WF)  \cite{WF}  method and proved that the WF algorithm can achieve linear convergence with $O(n \log n)$ Gaussian random measurements. 
Lately, Chen and Cand\`es improved the result to $O(n)$ Gaussian random measurements by incorporating a truncation, namely the Truncated Wirtinger Flow (TWF) \cite{TWF} algorithm.
Other non-convex methods with provable guarantees include the Gauss-Newton \cite{Gaoxu}, the trust-region \cite{turstregion}, Smoothed Amplitude Flow \cite{cai2021solving}, Truncated Amplitude Flow (TAF) algorithm \cite{TAF},  Reshaped Wirtinger Flow (RWF) \cite{RWF} algorithm and Perturbed Amplitude Flow (PAF) \cite{PAF} algorithm, to name just a few. 
 We refer the reader to survey papers  \cite{jaganathan2016phase,shechtman2015phase} for accounts of recent developments in the theory, algorithms and applications of phase retrieval.

\subsubsection{Randomized Kaczmarz method for linear equations}
Kaczmarz method is one of the most popular algorithms for solving overdetermined system of linear equations \cite{karczmarz1937angenaherte}, which iteratively project the current estimate onto the hyperplane of chosen equation at a time.
Suppose the system of linear equations we want to solve is given by $A\vx=\vy$, where $A\in \C^{m\times n}$.  In each iteration of the Kaczmarz method,  one row $\va_{i_k}$ of $A$ is selected and then the new 
iterate $\vz_{k+1}$ is obtained by projecting the current estimate $\vz_k$ orthogonally onto the solution hyperplane of $\nj{\va_{i_k},\vz}=y_{i_k}$ as follows:
\begin{equation} \label{eq:linearkacz}
\vz_{k+1}= \vz_k+\frac{y_{i_k}-\nj{\va_{i_k},\vz_k}}{\norm{\va_{i_k}}^2} \va_{i_k}.
\end{equation}
The classical version of Kaczmarz method sweeps through the rows of A in a cyclic manner, however, it lacks useful theoretical guarantees. Existing results in this manner are based on quantities of matrix $A$ which are hard 
to compute \cite{deutsch1,deutsch2,Galntai}. 
In 2009, Strohmer and Vershyinin \cite{strohmer2009r} propose a randomized Kaczmarz method where the row of $A$ is selected in random order and they prove that 
this randomized Kaczmarz method  is convergent with expected exponential rate. More precisely, at each step $k$, if the index $i_k$ is chosen randomly from the $\dkh{1,\ldots,m}$ with probability proportional to $\norm{\va_{i_k}}^2$ then for any initial $\vz_0$, the iteration $\vz_k$ given by randomized Karzmarz scheme \eqref{eq:linearkacz} obeys 
\[
\E \norm{\vz_k-\vx} \le \xkh{1-\frac{1}{k(A)\cdot n}}^{k/2} \cdot  \norm{\vz_0-\vx},
\]
where $k(A)$ is the condition number of $A$.

\subsubsection{Randomized Kaczmarz method for phase retrieval}
As stated before, the randomized Kaczmarz method for phase retrieval is proposed by Wei in 2015.
He was able to show numerically \cite{Wei2015} that the randomized Kaczmarz method exhibits empirical performance over other state-of-the-art phase retrieval algorithms, 
but lack of adequate theoretical performance guarantee.
Lately, in the real case, when the measurements $\va_j$ are drawn independently and uniformly from the unit sphere, 
several results have been established independently to guarantee the linear convergence of the randomized Kaczmarz method under appropriate initialization.

For instance, Tan and Vershynin \cite{tan2019phase} prove that for any $0<\delta, \delta_0 \le 1$ if $m\gtrsim n\log(m/n) +\log(1/\delta_0)$ and $\va_j \in \R^n$ are drawn independently and uniformly from the unit sphere,
 then with probability at least $1-\delta_0$ it holds:
the $k$-th step randomized Kaczmarz estimate $\vz_k$ given by \eqref{eq:updateformula} satisfies
\[
\E_{\mathcal{I}^k} \zkh{\dist(\vz_{k}, \vx) \1_{\tau=\infty}} \le (1-\frac{1}{2n})^{k/2} \dist(\vz_0, \vx),
\]
provided $\dist(\vz_0, \vx) \le c\delta \norm{\vx}$ for some constant $c>0$.
Furthermore, the probability $\PP(\tau<\infty) \le \delta^2$. Here $\tau$ is the stopping time and $\E_{\mathcal{I}^k}$ denotes the expectation with respect to randomness $\mathcal{I}^k:=\dkh{i_1,i_2,\ldots,i_k}$
conditioned on the high probability event of random measurements $\dkh{\va_j}_{j=1}^m$.

\subsection{Our Contributions}
As stated before, randomized Kaczmarz method is a popular and convenient method for solving phase retrieval problem due to its fast convergence and low computational complexity.
For the real setting, the theoretical guarantee of linear convergence has been established, however, there is no result concerning the rate of convergence in the complex setting.
Since there is an essential difference between the real setting and complex setting, the convergence of randomized Kaczmarz method in the complex setting has been  left as a conjecture \cite[Section 7.2]{tan2019phase}. 
The goal of this paper is to give a positive answer to this conjecture, as shown below.

\begin{theorem} \label{th:mainresult}
 Assume that the measurement vectors $\va_1,\ldots,\va_m \in \C^n$ are drawn independently and uniformly from the unit sphere $\CS_{\C}^{n-1}$.
 For any $0<\delta<1$, let $\vz_0$ be an initial estimate to $\vx$ such that $ \dist(\vz_0, \vx) \le 0.01\delta \norm{\vx}$.  There exist universal constants $C_0,c_0>0$ such that if $m\ge C_0 n$ then
with probability at least $1-14\exp(-c_0 n)$ it holds: the iteration $\vz_k$ given by randomized Kaczmarz update rule
\eqref{eq:updateformula} obeys
\[
\E_{\mathcal{I}^k} \zkh{\dist(\vz_{k}, \vx) \1_{\tau=\infty}} \le (1-0.03/n)^{k/2} \dist(\vz_0, \vx),
\]
where $\tau$ is the stopping time defined by
\begin{equation*} 
\tau:= \min\dkh{k: \vz_k \notin B} \quad \mbox{with} \quad B:=\dkh{\vz: \dist(\vz, \vx) \le 0.01 \norm{\vx}}.
\end{equation*}
Furthermore, the probability $\PP(\tau<\infty) \le \delta^2$. Here $\E_{\mathcal{I}^k}$ denotes the expectation with respect to randomness $\mathcal{I}^k:=\dkh{i_1,i_2,\ldots,i_k}$
conditioned on the high probability event of random measurements $\dkh{\va_j}_{j=1}^m$.
\end{theorem}

The theorem asserts  that the randomized Kaczmarz method converges linearly to the global solution $\vx$ (up to a global phase) in expectation for random measurements $\va_j \in \C^n$ which are drawn independently and uniformly from the complex unit sphere, or equivalent are independent complex Gaussian random vectors, with an optimal sample complexity.

\begin{remark}
Theorem  \ref{th:mainresult} requires an initial estimate $\vz_0$ which is close to the target solution. 
In fact, a good initial estimate can be obtained easily by spectral initialization which is widely used in non-convex algorithms for phase retrieval.
For instance, when $\va_j \in \C^n$ are complex Gaussian random vectors,
Gao and Xu \cite{Gaoxu} develop a spectral method based on exponential function, and  prove that with probability at least $1-\exp(-cn)$ the
spectral initialization can give an initial guess $\vz_0$ satisfying $\dist(\vz_0, \vx) \le \epsilon \norm{\vx}$ for any fixed $\epsilon$, provided $m\ge Cn$ for a positive constant $C$.
 We refer the reader to \cite{TWF, TAF, RWF}  for others spectral initialization and \cite{luo2019optimal,Mondelli} for the optimal design of a spectral initialization.

\end{remark}

\subsection{Notations}
Throughout this paper, we assume the measurements $\va_j\in \C^n, \; j=1,\ldots,m $ are drawn independently and uniformly from the 
complex unit sphere.  We say $\xi \in \C^n$ is a complex Gaussian
random vector if $\xi \sim 1/\sqrt{2}\cdot \cN(0,I_n)+i/\sqrt{2}\cdot \cN(0,I_n)$.
We write $\vz \in \mathbb{S}_{\C}^{n-1}$ if $\vz \in \C^n$ and $\norm{\vz}=1$.
Let $\Re(z) \in \R$ and $\Im(z) \in \R$ denote the real and imaginary part of a complex number $z\in \C$.
 For any $A,B\in \R$, we use $ A \lesssim B$
to denote $A\le C_0 B$ where $C_0\in \R_+$ is an  absolute constant.  The notion
$\gtrsim$ can be defined similarly. 
 In this paper, we use  $C,c$ and the subscript
(superscript)   form of them to denote universal constants whose values vary with the
context.

\subsection{Organization}
The paper is organized as follows.
 In Section 2, we introduce some notations and definitions that will be used in our paper.  In particular, the restricted strong convexity condition plays a key role in the proof of main result.
In Section 3, we first show that under the restricted strong convexity condition, a convergence result for a single step can be established, and then we show the main result can be proved by using the tools from stochastic process. In Section 4, we demonstrate that the random measurements drawn independently and uniformly from the complex unit sphere satisfies restricted strong convexity condition with high probability.
A brief discussion is presented in Section 5.  Appendix collects the
technical lemmas needed in the proofs.

\section{Preliminaries}
The aim of this section is to introduce some definitions that will be used in our paper.
Let $\vx\in \C^n$ be the target signal we want to recover.
The measurements we obtain are
\begin{equation} \label{eq:measectime}
b_j=\abs{\nj{\va_j,\vx}},\quad j=1,\ldots,m,
\end{equation}
where $\va_j \in \C^n$ are measurement vectors.  In this paper, we assume without loss of generality that $\va_j \in \mathbb{S}_{\C}^{n-1}$ for all $j=1,\ldots,m$.
For the recovery of $\vx$ we consider the randomized Kaczmarz method given by 
\begin{equation}  \label{eq:kaczschem}
\vz_{k+1}=\vz_k-\xkh{1-\frac{b_{i_k}}{|\va_{i_k}^*\vz_k|}} \va_{i_k} \va_{i_k}^*\vz_k ,
\end{equation}
where $i_k$ is chosen uniformly from the $\dkh{1,\ldots,m}$ at random at the $k$-th iteration. 

Obviously, for any $\vz$ if  $\vz$ is a solution to \eqref{eq:measectime} then $\vz e^{i\phi}$ is also a solution to it for any $\phi\in \R$.
Thus the set of solutions to \eqref{eq:measectime}   is $\dkh{\vx e^{i\phi}: \phi \in \R}$, which is a one-dimensional circle in $\C^n$. For this reason, we define the distance between $\vz$ and $\vx$ as
\[
\dist(\vz,\vx)=\min_{\phi\in \R}\norm{\vz-\vx e^{i\phi}}.
\]
For convenience, we also define the phase $\phi(\vz)$ as
\begin{equation} \label{eq:defphi}
\phi(\vz):=\mbox{argmin}_{\phi \in \R} \norm{\vz-\vx e^{i\phi}}
\end{equation}
for any $\vz \in \C^n$. Moreover, for any $\epsilon \ge 0$ we define the $\epsilon$-neighborhood of $\vx$ as
\begin{equation} \label{eq:Eeps}
E(\epsilon):=\dkh{\vz\in \C^n: \dist(\vz,\vx) \le \epsilon}.
\end{equation}
The following  auxiliary loss function plays a key role in the proof of main result:
\begin{equation}\label{eq:loss}
f(\vz)=\frac{1}{m} \sum_{j=1}^m \xkh{\abs{\va_j^* \vz}-\abs{\va_j^* \vx}}^2.
\end{equation}
Since it is not differentiable, we shall need the directional derivative.
For any vector $\vv\neq 0$ in $\C^n$, the {\em one-sided directional derivative of $f$} at $\vz$ along the direction $\vv$ is given by
\[
       D_{\vv} f(\vz):=\lim_{t\to 0^+} \frac{f(\vz+t\vv)-f(\vz)}{t}
\]
if the limit exists. It is not difficult to compute that the one-sided directional derivative of $f$ in \eqref{eq:loss} along any direction $\vv$  is
\begin{equation} \label{eq:direcdiriva}
 D_{\vv} f(\vz)=\frac{2}{m} \sum_{j=1}^m \xkh{1-\frac{|\va_j^* \vx|}{|\va_j^* \vz|}} \Re(\va_j^* \vv \vz^* \va_i).
\end{equation}
Finally,  we need the assumption that $f$  satisfies a local restricted strong convexity on $E(\epsilon)$, which
essentially states that the function is well behaved along the line connecting the current point to its nearest global solution.
Here, $E(\epsilon)$ and $f$ are defined in \eqref{eq:Eeps} and \eqref{eq:loss}, respectively.

\begin{definition}[Restricted Strong Convexity]
The function $f$ is said to obey the restricted strong convexity RSC$(\gamma,\epsilon)$ for some $\gamma,\epsilon >0$ if
\[
 D_{\vz-\vx e^{i \phi(\vz)}} f(\vz) \ge \gamma \norm{\vz-\vx e^{i \phi(\vz)}}^2 + f(\vz)
\]
for all $\vz \in E(\epsilon)$. 
\end{definition}

\section{Proof of The Main Result}
In this section we present the detailed proof of the main result. We first prove that under the assumption of $f$ satisfying restricted strong convexity,
a bound for the expected decrement in distance to the solution set can be established for the randomized Kaczmarz scheme in a single step. 
Next, we show that for random measurements $\va_j \in \C^n, j=1,\ldots,m$ 
which are drawn independently and uniformly from the complex unit sphere,  the function $f$ defined in \eqref{eq:loss} 
satisfies the restricted strong convexity with high probability,  provided $m\ge Cn$ for some constant $C>0$. Finally, using the tools from
stochastic process, we could prove the randomized Kaczmarz method is linearly convergent in expectation, which concludes the proof of  the main result.

\begin{theorem} \label{eq:twoiteration}
 Assume that $f$ defined in \eqref{eq:loss}   satisfies  the restricted strong convexity RSC$(\gamma,\epsilon)$. 
 Then the iteration $\vz_{k+1}$ given by randomized Kaczmarz update rule \eqref{eq:kaczschem} obeys
\[
\E_{i_k} \zkh{\dist^2(\vz_{k+1}, \vx)} \le (1-\gamma) \dist^2(\vz_{k}, \vx)
\]
 for all $\vz_k$ satisfying $\dist(\vz_{k}, \vx) \le \epsilon \norm{\vx}$. Here, $\E_{i_k} $ denotes the expectation with respect to randomness of $i_k$ at iteration $k$.
\end{theorem}
\begin{proof}
Recognize that $\norm{\va_{i_k}}=1$. 
Using the restricted strong convexity condition RSC$(\gamma,\epsilon)$, we have
\begin{eqnarray*}
\E_{i_k}\dist^2(\vz_{k+1}, \vx) &=& \E_{i_k}\norm{\vz_{k+1}-\vx e^{i \phi(\vz_{k+1})}}^2 \\
& \le & \E_{i_k} \left\| \vz_{k}-\vx e^{i \phi(\vz_{k})} - \Big(1-\frac{y_{i_k}}{|\va_{i_k}^*\vx_k|}\Big) \va_{i_k}^*\vz_k \va_{i_k}  \right\|_2^2  \\
&=& \norm{\vz_{k}-\vx e^{i \phi(\vz_{k})} }^2+ \E_{i_k}  \Big(1-\frac{y_{i_k}}{|\va_{i_k}^*\vz_k|}\Big)^2 \abs{\va_{i_k}^*\vz_k}^2\\ 
&&  - 2 \E_{i_k} \Re\xkh{ \Big(1-\frac{y_{i_k}}{|\va_{i_k}^*\vz_k|}\Big) \vz_k^*  \va_{i_k} \va_{i_k}^* (\vz_{k}-\vx e^{i \phi(\vz_{k})} ) } \\
&=&  \norm{\vz_{k}-\vx e^{i \phi(\vz_{k})} }^2 +f(\vz_k) - D_{\vz_k-\vx e^{i \phi(\vz_k)}} f(\vz_k) \\
&\le & (1-\gamma) \norm{\vz_{k}-\vx e^{i \phi(\vz_{k})} }^2,
\end{eqnarray*}
where the third equation follows from the expression of directional derivative as shown in \eqref{eq:direcdiriva}.
This completes the proof.
\end{proof}

\begin{theorem} \label{th:RSCcond}
Assume the measurement vectors $\va_1,\ldots,\va_m \in \C^n$ are drawn uniformly from the unit sphere $\CS_{\C}^{n-1}$. 
Suppose that  $m\ge C_0 n$ and $f$ is defined in \eqref{eq:loss}. Then $f$ satisfies the restricted strong convexity RSC$(\frac {0.03}{n},0.01)$
with probability at least $1-14\exp(-c_0 n)$, where $C_0, c_0$ are universal positive constants.
\end{theorem}
\begin{proof}
The proof of this theorem is deferred to Section \ref{sec:Rsc}.
\end{proof}

Based on Theorem \ref{eq:twoiteration} and Theorem \ref{th:RSCcond}, we obtain that if $m\ge C_0n$ for some universal constant $C_0>0$,
then with probability at least $1-14\exp(-c_0 n)$, the $(k+1)$-th iteration obeys
\[
\E_{i_{k+1}} \zkh{\dist^2(\vz_{k+1}, \vx)} \le (1-0.03/n) ~\dist^2(\vz_{k}, \vx),
\] 
provided $\dist(\vz_{k}, \vx) \le 0.01 \norm{\vx}$ at $k$ step.
To be able to iterate this result recursively we need  the condition $\dist(\vz_{k}, \vx) \le 0.01 \norm{\vx}$ holds for all $k$,
however, it does not hold arbitrarily. Hence, we introduce a stopping time 
\begin{equation} \label{eq:stoppingtime}
\tau:= \min\dkh{k: \vz_k \notin B},
\end{equation}
where $B:=\dkh{\vz: \dist(\vz, \vx) \le 0.01 \norm{\vx}}$. With this in place, we can give the proof of Theorem \ref{th:mainresult}.
We restate our main result here for convenience.

\begin{theorem}
Suppose $m\ge C_0n$ for some universal constant $C_0>0$.
 Assume the measurement vectors $\va_1,\ldots,\va_m \in \C^n$ are drawn independently and uniformly from the unit sphere $\CS_{\C}^{n-1}$.
 For any $0<\delta<1$, let $\vz_0$ be an initial estimate to $\vx$ such that $\norm{\vz_0-\vx} \le 0.01\delta \norm{\vx}$.  Let $\tau$ be the stopping time
defined in \eqref{eq:stoppingtime}. Then with probability at least $1-14\exp(-c_0 n)$ it holds: the iteration $\vz_k$ given by randomized Kaczmarz update rule \eqref{eq:kaczschem} obeys
\[
\E_{\mathcal{I}^k} \zkh{\dist(\vz_{k}, \vx) \1_{\tau=\infty}} \le (1-0.03/n)^{k/2} \dist(\vz_0, \vx).
\]
Furthermore, the probability $\PP(\tau<\infty) \le \delta^2$. Here $\E_{\mathcal{I}^k}$ denotes the expectation with respect to randomness $\mathcal{I}^k:=\dkh{i_1,i_2,\ldots,i_k}$
conditioned on the high probability event of random measurements $\dkh{\va_j}_{j=1}^m$ and $c_0>0$ is a universal constant.
\end{theorem}
\begin{proof}
From Theorem \ref{eq:twoiteration} and Theorem \ref{th:RSCcond}, we obtain that 
if $m\ge C_0n$ then with probability at least $1-14\exp(-c_0 n)$ it holds
\begin{eqnarray*}
\E_{\mathcal{I}^{k+1}} \big[\dist^2(\vz_{k+1}, \vx) \1_{\tau >k+1} ~\big |~ \vz_k \in B \big] &\le &\E_{\mathcal{I}^{k+1}}\big[\dist^2(\vz_{k+1}, \vx) \1_{\tau >k} ~\big |~ \vz_k \in B \big] \\
&=& \E_{\mathcal{I}^{k+1}} \big[\dist^2(\vz_{k+1}, \vx)  ~\big |~ \vz_k \in B \big]\1_{\tau >k}\\
&\le &  (1-0.03/n) ~\dist^2(\vz_{k}, \vx) \1_{\tau >k}.
\end{eqnarray*}
Note that $\vz_k \in B$ is an event with respect to randomness $\mathcal{I}^{k}$. Taking expectation gives
\begin{eqnarray*}
\E_{\mathcal{I}^{k+1}}\zkh{\dist^2(\vz_{k+1}, \vx) \1_{\tau >k+1}}& =& \E_{\mathcal{I}^{k}} \zkh{\E_{i_{k+1}} \big[\dist^2(\vz_{k+1}, \vx) \1_{\tau >k+1} ~\big |~ \vz_k \in B \big]} \\
&\le &  (1-0.03/n) ~\E_{\mathcal{I}^{k}} \zkh{\dist^2(\vz_{k}, \vx) \1_{\tau >k}} . 
\end{eqnarray*}
By induction, we arrive at the first part of the conclusion.

For the second part, define $Y_k:= \norm{ \vz_{k \wedge \tau}-\vx}^2$. Using the similar idea of Theorem 3.1 in \cite{tan2019phase}, we can check that $Y_k$ is a non-negative supermartingale.
It then follows from the supermartingale maximum inequality that
\[
\PP\xkh{\sup_{1\le k<\infty} Y_k\ge 0.01^2 \norm{\vx}^2} \le \frac{\E Y_0}{0.01^2 \norm{\vx}^2} \le \delta^2.
\]
This completes the proof.
\end{proof}

\section{Proof of Theorem \ref{th:RSCcond}} \label{sec:Rsc}
\begin{proof}[Proof of Theorem \ref{th:RSCcond}]
For any $\vz\in \C^n$, set $\vh=e^{-i \phi(\vz)} \vz -\vx$ where $\phi(\vz)$ is defined in \eqref{eq:defphi}. 
It is easy to check that the function $f$ given in \eqref{eq:loss} can be rewritten as
\[
f(\vz)=\frac{1}{m} \sum_{j=1}^m \xkh{\abs{\va_j^* \vh}^2+2\Re(\vh^* \va_j \va_j^* \vx ) + 2\abs{\va_j^* \vx}^2-2\abs{\va_j^* \vz}\abs{\va_j^* \vx}  }.
\]
To show that the function $f$ satisfies the restricted strong convexity, from the definition, it suffices to  give a lower bound for  $D_{\vz-\vx e^{i \phi(\vz)}} f(\vz) -f(\vz)$.
Note that 
$$|\va_j^* \vz|^2=|\va_j^* \vx|^2+2\Re(\vh^* \va_j \va_j^* \vx )+|\va_j^* \vh|^2.$$
 By some algebraic computation, we immediately have
\begin{eqnarray}
 && D_{\vz-\vx e^{i \phi(\vz)}} f(\vz) -f(\vz) \nonumber \\
 &=& \frac{2}{m} \sum_{j=1}^m \xkh{1-\frac{|\va_j^* \vx|}{|\va_j^* \vz|}} \xkh{ \abs{\va_j^* \vh}^2+\Re(\vh^* \va_j \va_j^* \vx )}-\frac{1}{m} \sum_{j=1}^m \xkh{\abs{\va_j^* \vz}-\abs{\va_j^* \vx}}^2 \nonumber\\
 &=&  \frac{1}{m} \sum_{j=1}^m \abs{\va_j^* \vh}^2+ \frac{2}{m} \sum_{j=1}^m \xkh{\abs{\va_j^* \vz}\abs{\va_j^* \vx}- \abs{\va_j^* \vx}^2-\frac{|\va_j^* \vx| |\va_j^* \vh|^2}{|\va_j^* \vz|} - \frac{|\va_j^* \vx| \Re(\vh^* \va_j \va_j^* \vx ) }{|\va_j^* \vz|} } \nonumber \\
 &=&  \frac{1}{m} \sum_{j=1}^m \abs{\va_j^* \vh}^2+ \frac{2}{m} \sum_{j=1}^m \frac{ |\va_j^* \vx|^3-|\va_j^* \vz| |\va_j^* \vx|^2+ |\va_j^* \vx| \Re(\vh^* \va_j \va_j^* \vx )}{|\va_j^* \vz|}\nonumber \\
  &=&  \frac{1}{m} \sum_{j=1}^m \abs{\va_j^* \vh}^2+ \frac{2}{m} \sum_{j=1}^m \frac{|\va_j^* \vz| |\va_j^* \vx| \Re(\vh^* \va_j \va_j^* \vx )- |\va_j^* \vx|^2 \Re(\vh^* \va_j \va_j^* \vx )-|\va_j^* \vx|^2 |\va_j^* \vh|^2}{|\va_j^* \vz|(|\va_j^* \vz|+|\va_j^* \vx|)}\nonumber\\
    &=&  \frac{1}{m} \sum_{j=1}^m \abs{\va_j^* \vh}^2- \frac{2}{m} \sum_{j=1}^m \frac{|\va_j^* \vx|^2 |\va_j^* \vh|^2}{|\va_j^* \vz|(|\va_j^* \vz|+|\va_j^* \vx|)}+ \frac{2}{m} \sum_{j=1}^m \frac{|\va_j^* \vx| |\va_j^* \vh|^2 \Re(\vh^* \va_j \va_j^* \vx )}{|\va_j^* \vz|(|\va_j^* \vz|+|\va_j^* \vx|)^2} \nonumber \\
    &&+ \frac{4}{m} \sum_{j=1}^m \frac{ |\va_j^* \vx| \Re^2(\vh^* \va_j \va_j^* \vx )}{|\va_j^* \vz|(|\va_j^* \vz|+|\va_j^* \vx|)^2} \label{eq:exprestrc}.
\end{eqnarray}
We next divide the indexes into two groups: $j\in I_{\alpha}$ and $j \in I_{\alpha}^c$, where  $I_{\alpha}:=\dkh{j: |\va_j^* \vx| \ge \alpha |\va_j^* \vh|}$ for some fixed parameter $\alpha>0$.  
For convenience, we denote $D_{\vz-\vx e^{i \phi(\vz)}} f(\vz) -f(\vz):= \frac{1}{m} \sum_{j=1}^m T_j$.
We claim that for any $\alpha>1$ it holds
\begin{equation}  \label{eq:Tjin}
 T_j \ge  \frac{4\alpha^3}{(\alpha+1)(2\alpha+1)^2}\cdot  \frac{ \Re^2(\vh^* \va_j \va_j^* \vx )}{|\va_j^* \vx|^2}-\frac{8\alpha^2-5\alpha+1}{(\alpha-1)(2\alpha-1)^2}\cdot   |\va_j^* \vh|^2 \quad \mbox{for } j\in I_{\alpha} 
\end{equation}
and 
\begin{equation} \label{eq:Tjout}
 T_j \ge - 3\abs{\va_j^* \vh}^2 \quad \mbox{for } j\in I_{\alpha}^c .
\end{equation}
This taken collectively with the identity \eqref{eq:exprestrc} leads to a lower estimate
\begin{eqnarray}
 D_{\vz-\vx e^{i \phi(\vz)}} f(\vz) -f(\vz) &\ge&  \frac{4\alpha^3}{(\alpha+1)(2\alpha+1)^2}\cdot  \frac{1}{m} \sum_{j \in I_{\alpha}} \frac{ \Re^2(\vh^* \va_j \va_j^* \vx )}{|\va_j^* \vx|^2} \nonumber\\
 &&-\frac{8\alpha^2-5\alpha+1}{(\alpha-1)(2\alpha-1)^2}\cdot  \frac{1}{m} \sum_{j \in I_{\alpha}} |\va_j^* \vh|^2-   \frac{3}{m} \sum_{j \in I_{\alpha}^c} \abs{\va_j^* \vh}^2, \label{eq:Dzxfz}
\end{eqnarray}
leaving us with three quantities in the right-hand side to deal with.
Let $\rho:=\norm{\vh}$. From the definition of $\vh$, it is easy to check $\Im(\vh^* \vx)=0$.
According to Lemma \ref{le:hxratiox}, we immediately obtain that for any $0<\delta\le 1$ there exist universal constants $C,c>0$ such that 
 if $\alpha\rho \le 1/3$ and $m\ge C\delta^{-2} \log(1/\delta)   n$ then with probability at least $1-6\exp(-c \delta^2 n)$ it holds
\begin{equation} \label{eq:lowerhxovex2}
 \frac{1}{m} \sum_{j \in I_{\alpha}} \frac{ \Re^2(\vh^* \va_j \va_j^* \vx )}{|\va_j^* \vx|^2} \ge \frac 1n \cdot \xkh{\frac 3{8} - \frac {\alpha^2 \rho^2}{(0.99+\alpha \rho)^2} -\delta} \norm{\vh}^2.
\end{equation}
For the second term, it follows from Lemma \ref{le:ah2} that for $m\ge C\delta^{-2} n$, with probability at least $1-2\exp(-c\delta^2 n)$,
\begin{equation} \label{eq:lowerah2}
\frac{1}{m} \sum_{j=1}^m |\va_j^* \vh|^2 \le \frac{1+\delta}{n} \norm{\vh}^2.
\end{equation}
Finally, for the third term, applying Lemma \ref{le:ah2complem}, we have that when $m\ge C\delta^{-2} \log(1/\delta)   n$ and  $0<\alpha \rho \le 0.4$, with probability at least $1-6\exp(-c\delta^2 n)$,
\begin{equation} \label{eq:upperah2com}
\frac{1}{m} \sum_{j \in I_{\alpha}^c} \abs{\va_j^* \vh}^2 \le \frac 1n \cdot \xkh{\frac {2\alpha^2 \rho^2}{0.99+\alpha^2 \rho^2} +\delta} \norm{\vh}^2.
\end{equation}
Setting $\alpha:=12$,  $\delta:=0.001$ and putting \eqref{eq:lowerhxovex2}, \eqref{eq:lowerah2}, \eqref{eq:upperah2com} into \eqref{eq:Dzxfz}, we obtain the conclusion that
with probability at least $1-14\exp(-c_0 n)$ it holds
\[
 D_{\vz-\vx e^{i \phi(\vz)}} f(\vz) -f(\vz)  \ge \frac{0.03}{n} \norm{\vh}^2 \quad \mbox{for all}\quad  \norm{\vh} \le 0.01,
\]
 provided $m\ge C_0 n$. Here,  $C_0,c_0$ are universal positive constants.

It remains to prove the claims. We first consider the case where $j\in I_{\alpha}$. 
It follows from  \eqref{eq:exprestrc} that 
\[
T_j=  \abs{\va_j^* \vh}^2- \frac{2|\va_j^* \vx|^2 |\va_j^* \vh|^2}{|\va_j^* \vz|(|\va_j^* \vz|+|\va_j^* \vx|)}+  \frac{2|\va_j^* \vx| |\va_j^* \vh|^2 \Re(\vh^* \va_j \va_j^* \vx )}{|\va_j^* \vz|(|\va_j^* \vz|+|\va_j^* \vx|)^2} +  \frac{ 4|\va_j^* \vx| \Re^2(\vh^* \va_j \va_j^* \vx )}{|\va_j^* \vz|(|\va_j^* \vz|+|\va_j^* \vx|)^2}.
\]
From the definition of $I_{\alpha}$,  it is easy to see that when $j\in I_{\alpha}$ we have
\begin{equation} \label{eq:lowupperofI}
(1-1/\alpha) |\va_j^* \vx| \le |\va_j^* \vx|-|\va_j^* \vh|  \le  |\va_j^* \vz|\le  |\va_j^* \vx|+|\va_j^* \vh|   \le  (1+1/\alpha) |\va_j^* \vx|.
\end{equation}
Thus  the second term of $T_j$ obeys
\[
\frac{|\va_j^* \vx|^2 |\va_j^* \vh|^2 }{|\va_j^* \vz|(|\va_j^* \vz|+|\va_j^* \vx|)} \le \frac{\alpha^2}{(\alpha-1)(2\alpha-1)}|\va_j^* \vh|^2.
\]
Similarly, the third term of $T_j$ satisfies
\[
 \frac{|\va_j^* \vx| |\va_j^* \vh|^2 |\Re(\vh^* \va_j \va_j^* \vx )|}{|\va_j^* \vz|(|\va_j^* \vz|+|\va_j^* \vx|)^2}\le  \frac{|\va_j^* \vx|^2 |\va_j^* \vh|^3}{|\va_j^* \vz|(|\va_j^* \vz|+|\va_j^* \vx|)^2}\le  \frac{\alpha^2}{(\alpha-1)(2\alpha-1)^2}|\va_j^* \vh|^2.
\]
Finally, using the upper bound in \eqref{eq:lowupperofI}, we have
\begin{eqnarray*}
\frac{|\va_j^* \vx| \Re^2(\vh^* \va_j \va_j^* \vx )}{|\va_j^* \vz|(|\va_j^* \vz|+|\va_j^* \vx|)^2}& = & \frac{|\va_j^* \vx|^3}{|\va_j^* \vz|(|\va_j^* \vz|+|\va_j^* \vx|)^2} \cdot \frac{ \Re^2(\vh^* \va_j \va_j^* \vx )}{|\va_j^* \vx|^2}\\
&\ge& \frac{\alpha^3}{(\alpha+1)(2\alpha+1)^2}\cdot \frac{ \Re^2(\vh^* \va_j \va_j^* \vx )}{|\va_j^* \vx|^2}.
\end{eqnarray*}
Collecting the above three estimators, we have
\begin{eqnarray*} 
T_j &\ge&  \frac{4\alpha^3}{(\alpha+1)(2\alpha+1)^2}\cdot  \frac{ \Re^2(\vh^* \va_j \va_j^* \vx )}{|\va_j^* \vx|^2}-\frac{8\alpha^2-5\alpha+1}{(\alpha-1)(2\alpha-1)^2}\cdot   |\va_j^* \vh|^2,
\end{eqnarray*}
which proves the claim \eqref{eq:Tjin}.

We next turn to consider the case where $j\notin I_{\alpha}$. From the definition, we know $T_j$ can be denoted as
\[
T_j= 2 \xkh{1-\frac{|\va_j^* \vx|}{|\va_j^* \vz|}}  \Re( e^{-i \phi } \vh^* \va_j \va_j^* \vz ) - \xkh{|\va_j^* \vz|-|\va_j^* \vx|}^2  \quad \mbox{for all} \quad j.
\]
It then immediately  gives 
\begin{eqnarray*}
|T_j|  \le  \frac{ 2 \big| |\va_j^* \vz|-|\va_j^* \vx| \big | }{|\va_j^* \vz| }\cdot  |\vh^* \va_j \va_j^* \vz | +\xkh{|\va_j^* \vz|-|\va_j^* \vx|}^2 \le 3\abs{\va_j^* \vh}^2,
\end{eqnarray*}
where we use the Cauchy-Schwarz inequality and the fact that $\big| |\va_j^* \vz|-|\va_j^* \vx| \big | \le |\va_j^* \vh|$ in the last inequality. 
This completes the claim \eqref{eq:Tjout}.

\end{proof}

\section{Discussions}
This paper considers the convergence of randomized Kaczmarz method for phase retrieval in the complex setting. A linear convergence rate has been established by combining restricted strong convexity condition and tools from stochastic process, which gives a positive answer for the conjecture given in  \cite[Section 7.2]{tan2019phase}.

There are some interesting problems for future research. First, it has been shown numerically that randomized Karzmarz method is also efficient for solving Fourier phase retrieval problem, at least when the measurements follow the coded diffraction pattern (CDP) model, it is of practical interest to provide some theoretical guarantees for it.  Second, the convergence of randomized Kaczmarz method relies on a spectral initialization. Some numerical evidence have shown that randomized Kaczmarz method  works well even if we start from an arbitrary initialization. It is interesting to provide some theoretical justifications for it.

\section{Appendix }
\begin{lemma} \label{le:ah2}
 Suppose that the vectors $\va_1,\ldots,\va_m \in \C^n$ are drawn uniformly from the unit sphere $\CS_{\C}^{n-1}$. For any  $0<\delta \le 1$, if $m\ge C\delta^{-2} n$ then
with probability at least $1-2\exp(-c \delta^2 m)$  it holds
\[
\left \| \frac{1}{m} \sum_{j=1}^m \va_j \va_j^* -\frac 1n \cdot I \right \|_2 \le \frac {\delta} n.   
\]
Here, $C$ and $c$ are universal positive constants.
\end{lemma}
\begin{proof}
Assume that $\mathcal{N}$ is an $1/4$-net of the complex unit sphere $\CS_{\C}^{n-1}\subset \C^n$. It then follows from \cite[Lemma 4.4.3]{Vershynin2018} that
\[
\left \| \frac{1}{m} \sum_{j=1}^m \va_j \va_j^* -\frac 1n \cdot I \right \|_2 \le 2\max_{\vh \in \mathcal{N}} \abs{ \frac{1}{m} \sum_{j=1}^m | \va_j^* \vh|^2 -\frac 1n }.
\]
Here, the cardinality $|\mathcal{N} | \le 9^{2n}$.  Due to the unitary invariance of $\va_j$, for any fixed $\vh \in \CS_{\C}^{n-1}$,  we have
 \[
 \E|\va_j^* \vh|^2 =\E|\va_j^* \ve_1|^2=\E|\va_j^* \ve_2|^2=\cdots=\E|\va_j^* \ve_n|^2=\frac 1n \E\norm{\va_j}^2 =\frac 1n.
 \]
It means that for any fixed $\vh \in \CS_{\C}^{n-1}\subset \C^n $, the terms $|\va_j^* \vh|^2-1/n$ are  independent, mean zero,  sub-exponential random variables with sub-exponential  norm bounded by $K=c_1/n$ for some universal constant $c_1>0$ \cite[Theorem 3.4.6]{Vershynin2018}. Using Bernstein's inequality, we obtain that for any $0<\delta\le 1$ with probability at least $1-2\exp(- c_2\delta^2 m)$,
\begin{equation*}
\abs{\frac{1}{m} \sum_{j=1}^m |\va_j^* \vh|^2- \frac1n} \le \frac{\delta}{2n}
\end{equation*}
holds for some positive constant $c_2$. Taking the union bound over $\mathcal{N}$, we obtain that
\[
\left \| \frac{1}{m} \sum_{j=1}^m \va_j \va_j^* -\frac 1n \cdot I \right \|_2 \le \frac{\delta}{n}
\]
holds with probability at least 
\[
1-2\exp(- c_2\delta^2 m)\cdot 9^{2n} \ge 1-2\exp(- c\delta^2 m),
\]
provided $m\ge C\delta^{-2} n$ for some constants $C, c>0$.
This completes the proof.

\end{proof}


\begin{lemma} \label{le:hxratiox}
Let $\vx$ be a vector in $ \C^n$ with $\norm{\vx}=1$ and  $\lambda \ge 3$. 
Assume that the vectors $\va_1,\ldots,\va_m \in \C^n$ are drawn uniformly from the unit sphere $\CS_{\C}^{n-1}$. 
For any fixed $0< \delta\le 1$,  there exist universal constants 
$C, c>0$ such that if $m\ge C \delta^{-2} \log(1/\delta)  n$  then with probability at least $1-6\exp(-c\delta^2 n)$ it holds 
\[
 \frac{1}{m} \sum_{j=1}^m  \frac{ \Re^2(\vh^* \va_j \va_j^* \vx )}{|\va_j^* \vx|^2} \cdot \1_{\dkh{ \lambda |\va_j^* \vx|\ge |\va_j^* \vh|}} \ge \frac 1n \cdot \xkh{\frac 3{8} - \frac 1{(1+0.99\lambda)^2} -\delta} 
\]
for all $\vh \in \C^n$ with $\norm{\vh}=1$ and $\Im(\vh^* \vx)=0$. 
\end{lemma}
\begin{proof}
We first prove the result for any fixed $\vh$ and then apply an $\varepsilon$-net argument to develop a uniform bound for it.
To begin with, we introduce a series of auxiliary random Lipschitz functions to approximate the indicator functions. For any $j=1,\ldots,m$, define
\[
\chi_j(t):=\left\{ \begin{array}{ll} 
                         1,& \mbox{if} \quad t \le 0.99\lambda |\va_j^* \vx|; \\
                         -\frac{100}{\lambda |\va_j^* \vx|} t+100 ,& \mbox{if} \quad  0.99\lambda |\va_j^* \vx| \le t \le \lambda |\va_j^* \vx| ;\\
                         0,& \mbox{otherwise}.
                         \end{array}
                         \right.
\]
It then gives 
\begin{equation}\label{eq:chi}
 \frac{ \Re^2(\vh^* \va_j \va_j^* \vx )}{|\va_j^* \vx|^2} \cdot \1_{\dkh{ \lambda |\va_j^* \vx|\ge |\va_j^* \vh|}} \ge \frac{ \Re^2(\vh^* \va_j \va_j^* \vx )}{|\va_j^* \vx|^2} \chi_j(|\va_j^* \vh|) \ge  \frac{ \Re^2(\vh^* \va_j \va_j^* \vx )}{|\va_j^* \vx|^2} \cdot \1_{\dkh{ 0.99\lambda |\va_j^* \vx|\ge |\va_j^* \vh|}}.
\end{equation}
For any fixed $\vh$, since $\va_1,\ldots,\va_m $ are random vectors uniformly distributed on the unit sphere, it means that  the terms $ \frac{ \Re^2(\vh^* \va_j \va_j^* \vx )}{|\va_j^* \vx|^2} \chi_j(|\va_j^* \vh|) $ are independent sub-exponential random variables with the maximal sub-exponential  norm $K=c_1/n$ for some universal constant $c_1>0$ \cite[Theorem 3.4.6]{Vershynin2018}.
Apply Bernstein's inequality gives that for any fixed $0<\delta\le 1$ the following holds
\begin{equation}  \label{eq:fixhlowb}
 \frac{1}{m} \sum_{j=1}^m  \frac{ \Re^2(\vh^* \va_j \va_j^* \vx )}{|\va_j^* \vx|^2} \chi_j(|\va_j^* \vh|) \ge \E \xkh{\frac{ \Re^2(\vh^* \va \va^* \vx )}{|\va^* \vx|^2} \chi_j(|\va^* \vh|) } - \frac{\delta}{4n}
\end{equation}
with probability at least  $1-2\exp(- c_2\delta^2 m)$, where $c_2$ is a universal positive constant. Here, $\va \in \C^n$ is a vector uniformly distributed on the unit sphere.

Next, we give a uniform bound for the estimate \eqref{eq:fixhlowb}.  Construct an $\varepsilon$-net $\mathcal{N}$ over the unit sphere in $\C^n$ with cardinality $|\mathcal{N}| \le (1+\frac{2}{\varepsilon})^{2n}$.
Then we have
\[
 \frac{1}{m} \sum_{j=1}^m  \frac{ \Re^2(\vh^* \va_j \va_j^* \vx )}{|\va_j^* \vx|^2} \chi_j(|\va_j^* \vh|)  \ge \E \xkh{\frac{ \Re^2(\vh^* \va \va^* \vx )}{|\va^* \vx|^2} \chi_j(|\va^* \vh|) }-\frac{\delta}{4n} \quad \mbox{for all} \quad \vh \in \mathcal{N}
\]
with probability at least 
\[
1-2\exp(- c_2\delta^2 m) \cdot (1+\frac{2}{\varepsilon})^{2n}.
\]
For any $\vh$ with $\norm{\vh}=1$, there exists a $\vh_0 \in \mathcal{N}$ such that $\norm{\vh-\vh_0}\le \varepsilon$.
We claim that there exist universal constants $C', c_3>0$ such that if $m\ge C' n$ then with probability at least $1-2\exp(- c_3 m)$ it holds
\begin{equation} \label{eq:claimlipscthz}
\left| \frac{1}{m} \sum_{j=1}^m  \frac{ \Re^2(\vh^* \va_j \va_j^* \vx )}{|\va_j^* \vx|^2} \chi_j(|\va_j^* \vh|)-\frac{1}{m} \sum_{j=1}^m  \frac{ \Re^2(\vh_0^* \va_j \va_j^* \vx )}{|\va_j^* \vx|^2} \chi_j(|\va_j^* \vh_0|)   \right|  \le \frac{205\varepsilon}{n}.
\end{equation}
Choosing $\varepsilon:=\delta/820$,  we then obtain that
\begin{equation} \label{eq:unibojnd}
 \frac{1}{m} \sum_{j=1}^m  \frac{ \Re^2(\vh^* \va_j \va_j^* \vx )}{|\va_j^* \vx|^2} \chi_j(|\va_j^* \vh|)  \ge \E \xkh{\frac{ \Re^2(\vh^* \va \va^* \vx )}{|\va^* \vx|^2} \chi_j(|\va^* \vh|) }-\frac{\delta}{2n} \quad \mbox{for all} \quad \norm{\vh}=1
\end{equation}
holds with probability at least 
\[
1-2\exp(- c_3 m)-2\exp(- c_2 \delta^2 m)(1+\frac{2}{\varepsilon})^{2n} \ge 1-4\exp(- c_4 \delta^2 m),
\]
provided $m\ge C  \log(1/\delta) \delta^{-2} n$ for some positive constant $C$. Here  $c_4$ is a universal positive constant. To give a lower bound for the expectation in \eqref{eq:unibojnd},  recognize that if  $\xi \in \C^n$ is a complex Gaussian random vector then  $\xi/\norm{\xi}$ is a vector uniformly distributed on the unit sphere.
Since  $\norm{\xi} \le (1+\delta_1) \sqrt{n}$ holds for any fixed $0<\delta_1\le 1$ \cite[Theorem 3.1.1]{Vershynin2018}  with probability at least $1-2\exp(- c_5\delta_1^2 n)$ for some universal constant $c_5>0$, it then follows from Lemma \ref{le:hxratioxexpec} that
\begin{eqnarray*}
\E \xkh{\frac{ \Re^2(\vh^* \va \va^* \vx )}{|\va^* \vx|^2} \cdot \1_{\dkh{\lambda |\va^* \vx|\ge |\va^* \vh|}}} &=& \E \xkh{\frac{ \Re^2(\vh^* \xi \xi^* \vx )}{\norm{\xi}^2 |\xi^* \vx|^2} \cdot \1_{\dkh{\lambda |\xi^* \vx|\ge |\xi^* \vh|}}} \nonumber \\
&\ge &\frac{1}{(1+3\delta_1)n}\cdot   \E \xkh{\frac{ \Re^2(\vh^* \xi \xi^* \vx )}{ |\xi^* \vx|^2} \cdot \1_{\dkh{\lambda |\xi^* \vx|\ge |\xi^* \vh|}}} \nonumber\\
&\ge& \frac{1}{(1+3\delta_1)n}\cdot \xkh{ \frac 38 -\frac 1{(\lambda+1)^2}}.
\end{eqnarray*}
Taking $\delta_1:= \delta/2$, we obtain that  for any $\lambda\ge 2.95$ with probability at least $1-2\exp(- c_6 \delta^2 n)$ it holds
\begin{equation} \label{eq:claexpe}
\E \xkh{\frac{ \Re^2(\vh^* \va \va^* \vx )}{|\va^* \vx|^2} \cdot \1_{\dkh{\lambda |\va^* \vx|\ge |\va^* \vh|}}}  \ge \frac 1n \cdot \xkh{\frac 3{8} - \frac 1{(\lambda+1)^2} -\frac{\delta}{2}},
\end{equation}
where $c_6>0 $ is a universal constant.
Collecting  \eqref{eq:chi}, \eqref{eq:unibojnd} and \eqref{eq:claexpe} together, we obtain the conclusion that for any $\lambda\ge 3$ with probability at least $1-6\exp(- c \delta^2 n)$ it holds
\[
 \frac{ \Re^2(\vh^* \va_j \va_j^* \vx )}{|\va_j^* \vx|^2} \cdot \1_{\dkh{ \lambda |\va_j^* \vx|\ge |\va_j^* \vh|}} \ge \frac 1n \cdot \xkh{\frac 3{8} - \frac 1{(1+0.99 \lambda)^2} -\delta},
\]
provided $m\ge C  \log(1/\delta) \delta^{-2} n$. Here, $c$ is a universal positive constant.

Finally, it remains to prove the claim \eqref{eq:claimlipscthz}.  To this end, we claim that for all $j=1,\ldots,m$ it holds
\begin{eqnarray}
&& \left|  \frac{ \Re^2(\vh^* \va_j \va_j^* \vx )}{|\va_j^* \vx|^2} \chi_j(|\va_j^* \vh|)-  \frac{ \Re^2(\vh_0^* \va_j \va_j^* \vx )}{|\va_j^* \vx|^2} \chi_j(|\va_j^* \vh_0|)   \right|   \vspace{2em}  \nonumber\\ 
&\le& 101  |\va_j^* \vh| |\va_j^*( \vh-\vh_0)|+   101|\va_j^* \vh_0| |\va_j^*( \vh-\vh_0)|.  \label{eq:claimdiffer}
\end{eqnarray}
Indeed, from the definition of $\chi_j(t)$, if  both $|\va_j^* \vh| >\lambda |\va_j^* \vx|$ and $|\va_j^* \vh_0| >\lambda |\va_j^* \vx|$ then the above inequality holds directly. Thus, we only need to consider the case where  $|\va_j^* \vh| \le \lambda |\va_j^* \vx|$ or $|\va_j^* \vh_0| \le \lambda |\va_j^* \vx|$. Without loss of generality, we assume $|\va_j^* \vh| \le \lambda |\va_j^* \vx|$. Then we have
\begin{eqnarray*}
&& \left|  \frac{ \Re^2(\vh^* \va_j \va_j^* \vx )}{|\va_j^* \vx|^2} \chi_j(|\va_j^* \vh|)-  \frac{ \Re^2(\vh_0^* \va_j \va_j^* \vx )}{|\va_j^* \vx|^2} \chi_j(|\va_j^* \vh_0|)   \right| \\
&\le &  |\va_j^* \vh|^2 \abs{ \chi_j(|\va_j^* \vh|)-\chi_j(|\va_j^* \vh_0|)} +   \xkh{|\va_j^* \vh|+|\va_j^* \vh_0|} |\va_j^*( \vh-\vh_0)| \\
&\le &  \frac{100 |\va_j^* \vh|^2}{\lambda  |\va_j^* \vx|} |\va_j^*( \vh-\vh_0)|+   \xkh{|\va_j^* \vh|+|\va_j^* \vh_0|} |\va_j^*( \vh-\vh_0)|\\
&\le & 101  |\va_j^* \vh| |\va_j^*( \vh-\vh_0)|+   |\va_j^* \vh_0| |\va_j^*( \vh-\vh_0)| ,
\end{eqnarray*}
which gives \eqref{eq:claimdiffer}. According to Lemma \ref{le:ah2}, we obtain that for $m\ge C' n$ with probability at least $1-2\exp(- c_3 n)$ it holds
\begin{eqnarray*}
&& \left| \frac{1}{m} \sum_{j=1}^m  \frac{ \Re^2(\vh^* \va_j \va_j^* \vx )}{|\va_j^* \vx|^2} \chi_j(|\va_j^* \vh|)-\frac{1}{m} \sum_{j=1}^m  \frac{ \Re^2(\vh_0^* \va_j \va_j^* \vx )}{|\va_j^* \vx|^2} \chi_j(|\va_j^* \vh_0|)   \right| \\
&\le & \frac{101}{m} \sum_{j=1}^m  |\va_j^* \vh| |\va_j^*( \vh-\vh_0)|+  \frac{101}{m} \sum_{j=1}^m |\va_j^* \vh_0| |\va_j^*( \vh-\vh_0)| \\
&\le & 101 \sqrt{ \frac{1}{m} \sum_{j=1}^m  |\va_j^* \vh|^2}\sqrt{ \frac{1}{m} \sum_{j=1}^m   |\va_j^*( \vh-\vh_0)|^2}+101\sqrt{ \frac{1}{m} \sum_{j=1}^m  |\va_j^* \vh_0|^2}\sqrt{ \frac{1}{m} \sum_{j=1}^m   |\va_j^*( \vh-\vh_0)|^2} \\
&\le & \frac{205\varepsilon}{n},
\end{eqnarray*}
which proves the claim \eqref{eq:claimlipscthz}.

\end{proof}

\begin{lemma} \label{le:ah2complem}
Let $\vx$ be a vector in $ \C^n$ with $\norm{\vx}=1$ and  $0<\lambda \le 0.4 $. 
Assume that the vectors $\va_1,\ldots,\va_m \in \C^n$ are drawn uniformly from the unit sphere $\CS_{\C}^{n-1}$.  For any fixed $0<\delta\le 1$,  there exist universal constants $C, c>0$ such that for $m\ge C  \delta^{-2}\log(1/\delta) n$, with probability at least $1-6\exp(-c\delta^2 n)$,  it holds
\[
 \frac{1}{m} \sum_{j=1}^m  |\va_j^* \vh|^2 \cdot \1_{\dkh{  |\va_j^* \vx| \le \lambda |\va_j^* \vh|}} \le \frac{2\lambda^2}{(\lambda^2+0.99)n}+\frac{\delta}{n}
\]
for all  $\vh \in \C^n$ with $\norm{\vh}=1$ and $\Im(\vh^* \vx)=0$. 
\end{lemma}
\begin{proof}
Due to the non-Lipschitz of indicator functions, we introduce a series of auxiliary random Lipschitz functions to approximate them. For any $j=1,\ldots,m$, define
\[
\chi_j(t):=\left\{ \begin{array}{ll} 
                         t,& \mbox{if} \quad t \ge |\va_j^* \vx|^2/\lambda^2 ; \\
                         100t-\frac{99 |\va_j^* \vx|^2}{\lambda^2 } ,& \mbox{if} \quad  0.99|\va_j^* \vx|^2/\lambda^2 \le t \le |\va_j^* \vx|^2/\lambda^2;\\
                         0,& \mbox{otherwise}.
                         \end{array}
                         \right.
\]
Then it is easy to check that
\begin{equation}\label{eq:chithesec}
  \frac{1}{m} \sum_{j=1}^m  |\va_j^* \vh|^2 \cdot \1_{\dkh{  |\va_j^* \vx| \le \lambda |\va_j^* \vh|}}  \le  \frac{1}{m} \sum_{j=1}^m  \chi_j(|\va_j^* \vh|^2) \le  \frac{1}{m} \sum_{j=1}^m  |\va_j^* \vh|^2 \cdot \1_{\dkh{ 0.99 |\va_j^* \vx| \le \lambda |\va_j^* \vh|}}.
\end{equation}
For any fixed $\vh$, since $\va_1,\ldots,\va_m $ are drawn uniformly from the unit sphere $\CS_{\C}^{n-1}$, thus  the terms $\chi_j(|\va_j^* \vh|^2) $ are independent sub-exponential random variables with the maximal sub-exponential  norm $K=c_1/n$ for some universal constant $c_1>0$. According to Bernstein's inequality, for any fixed $0<\delta\le 1$, with probability at least  $1-2\exp(- c_2\delta^2 m)$,  it holds
\begin{equation} \label{eq:bern2}
\abs{  \frac{1}{m} \sum_{j=1}^m  \chi_j(|\va_j^* \vh|^2) -\E\zkh{ \chi_j(|\va^* \vh|^2)  }} \le \frac{\delta}{4n},
\end{equation}
where $c_2$ is a universal positive constant. Here, $\va \in \C^n$ is a vector uniformly distributed on the unit sphere.

To give a uniform bound for the the estimate \eqref{eq:bern2}, we  construct an $\varepsilon$-net $\mathcal{N}$ over the unit sphere in $\C^n$ with cardinality $|\mathcal{N}| \le (1+\frac{2}{\varepsilon})^{2n}$.
Then for any $\vh$ with $\norm{\vh}=1$, there exists a $\vh_0 \in \mathcal{N}$ such that $\norm{\vh-\vh_0}\le \varepsilon$.
Note that $\chi_j(t)$ is a Lipschitz function with Lipschitz  constant $100$. It then follows from  Lemma \ref{le:ah2} that for $m\ge C' n$ with probability at least $1-2\exp(- c_3 m)$ it holds
\begin{eqnarray*}
&&\abs{  \frac{1}{m} \sum_{j=1}^m  \chi_j(|\va_j^* \vh|^2) -  \frac{1}{m} \sum_{j=1}^m \chi_j(|\va_j^* \vh_0|^2) } \\
&\le &  \frac{100}{m} \sum_{j=1}^m  \xkh{|\va_j^* \vh|+|\va_j^* \vh_0|} |\va_j^* (\vh-\vh_0)|\\
&\le & 100 \sqrt{ \frac{1}{m} \sum_{j=1}^m |\va_j^* \vh|^2} \sqrt{ \frac{1}{m} \sum_{j=1}^m |\va_j^* (\vh-\vh_0)|^2}  +100 \sqrt{ \frac{1}{m} \sum_{j=1}^m |\va_j^* \vh_0|^2} \sqrt{ \frac{1}{m} \sum_{j=1}^m |\va_j^* (\vh-\vh_0)|^2} \\
&\le & \frac{202 \varepsilon}{n},
\end{eqnarray*}
where the third line follows from the Cauchy-Schwarz inequality.
Choosing $\varepsilon:=\delta/808$ and taking the union bound over $\mathcal{N}$, we obtain that 
\begin{equation} \label{eq:unibojnd2le}
 \frac{1}{m} \sum_{j=1}^m  \chi_j(|\va_j^* \vh|^2) \ge \E \xkh{ \frac{1}{m} \sum_{j=1}^m  \chi_j(|\va^* \vh|^2) }-\frac{\delta}{2n} \quad \mbox{for all} \quad \norm{\vh}=1
\end{equation}
holds with probability at least 
\[
1-2\exp(- c_3 m)-2\exp(- c_2\delta^2 m)(1+\frac{2}{\varepsilon})^{2n} \ge 1-4\exp(- c_4 \delta^2 m),
\]
provided $m\ge C  \delta^{-2}\log(1/\delta) n$ for some positive constants $C, c_4$.

Finally, we need to lower bound the expectation. To this end,  let $\xi \in \C^n$ be a complex Gaussian random vector.  We claim that for any $0<\lambda \le \sqrt{\frac{5-\sqrt{21}}{2}}$ it holds
 \begin{equation} \label{eq:exahsqu}
 \E \xkh{|\xi^* \vh|^2 \cdot \1_{\dkh{  |\xi^* \vx| \le \lambda |\xi^* \vh|}}} \le \frac{2\lambda^2}{\lambda^2+1}.
 \end{equation}
Since $\xi/\norm{\xi}$ is a vector uniformly distributed on the unit sphere and $\norm{\xi} \le (1-\delta_0) \sqrt{n}$ holds  for any fixed $0<\delta_0 \le 1$ with probability at least $1-2\exp(- c_5\delta_0^2 n)$ \cite[Theorem 3.1.1]{Vershynin2018}.
It then gives that for any $0<\lambda \le \sqrt{\frac{5-\sqrt{21}}{2}}$ with probability at least $1-2\exp(- c_5\delta_0^2 n)$ we have
 \begin{eqnarray}
 \E \xkh{ |\va^* \vh|^2 \cdot \1_{\dkh{  |\va^* \vx| \le \lambda |\va^* \vh|}} } &= &\E \xkh{\frac{|\xi^* \vh|^2}{\norm{\xi}^2} \cdot \1_{\dkh{  |\xi^* \vx| \le \lambda |\xi^* \vh|}}}  \nonumber\\
 &\le & \frac{1}{(1-\delta_0)n} \cdot  \E \xkh{|\xi^* \vh|^2 \cdot \1_{\dkh{  |\xi^* \vx| \le \lambda |\xi^* \vh|}}} \nonumber\\
 &\le &  \frac{1}{(1-\delta_0)n} \cdot  \frac{2\lambda^2}{\lambda^2+1}.  \label{eq:exahabssq}
 \end{eqnarray}
 Taking the constant $\delta_0 = \delta/3$, it then follows from \eqref{eq:chithesec}, \eqref{eq:unibojnd2le} and \eqref{eq:exahabssq} that for any fixed $0<\lambda \le 0.4$,  with  probability at least $1-6\exp(- c\delta^2 n)$, it holds
 \[
 \frac{1}{m} \sum_{j=1}^m  |\va_j^* \vh|^2 \cdot \1_{\dkh{  |\va_j^* \vx| \le \lambda |\va_j^* \vh|}} \le  \frac{2\lambda^2}{(\lambda^2+0.99)n}+\frac{\delta}{n},
 \]
provided $m\ge C  \delta^{-2}\log(1/\delta) n$, where $c$ is a universal positive constant. This completes the proof.

It remains to prove the claim \eqref{eq:exahsqu}.  Indeed, due to the unitary invariance of Gaussian random vector, without loss of generality, we assume $\vh=\ve_1$ and $\vx=\sigma \ve_1 + \tau e^{i \phi} \ve_2$, where $\sigma=\vh^*\vx \in \R, \abs{\sigma} \le 1$ and $\tau=\sqrt{1-\sigma^2} $.
Let $\xi_1, \xi_2$ be the first and second entries of $\xi$. Denote $\xi_1=\xi_{1,\Re}+ i \xi_{1,\Im}$ and $\xi_2=\xi_{2,\Re}+ i \xi_{2,\Im}$ where $\xi_{1,\Re}, \xi_{1,\Im}, \xi_{2,\Re}, \xi_{2,\Im} $
are independent Gaussian random variables with distribution  $\mathcal{N}(0,1/2)$. Then the inequality $ |\xi^* \vx| \le \lambda |\xi^* \vh|$ is equivalent to 
\begin{eqnarray*}
\xkh{\sigma\xi_{1,\Re} +\tau (\cos\phi \xi_{2,\Re}+\sin\phi \xi_{2,\Im}) }^2+\xkh{\sigma\xi_{1,\Im} -\tau (\sin\phi \xi_{2,\Re}-\cos\phi \xi_{2,\Im}) }^2  \le \lambda(\xi_{1,\Re}^2+\xi_{1,\Im}^2 ).
\end{eqnarray*}
To prove the inequality \eqref{eq:exahsqu}, we take the polar coordinates transformations and denote
\[
\left.
\begin{array}{r}
\xi_{1,\Re}= r_1 \cos \theta_1 \\
\xi_{1,\Im}= r_1 \sin \theta_1 \\
\sigma\xi_{1,\Re} +\tau (\cos\phi \xi_{2,\Re}+\sin\phi \xi_{2,\Im}) =r_2 \cos \theta_2\\
\sigma\xi_{1,\Im} -\tau (\sin\phi \xi_{2,\Re}-\cos\phi \xi_{2,\Im})= r_2 \sin \theta_2
\end{array} \right\}
\]
with $r_1, r_2 \in (0,+\infty), \theta_1,\theta_2 \in [0,2\pi]$. Then the expectation can be written as
\begin{eqnarray*}
G(\lambda,\sigma)&:=& \E \xkh{|\xi^* \vh|^2 \cdot \1_{\dkh{  |\xi^* \vx| \le \lambda |\xi^* \vh|}}} \\
&=&\frac{1}{\pi^2} \int_{0}^{2\pi} \int_{0}^{2\pi} \int_{0}^{+\infty} \int_{0}^{\lambda\cdot r_1} \frac{r_1^3 r_2}{\tau^2} e^{-(r_1^2+r_2^2)/\tau^2} \cdot e^{2\sigma r_1r_2 \cos(\theta_1-\theta_2)/\tau^2} d r_2 d r_1 d\theta_1d\theta_2.
\end{eqnarray*}
It gives 
\begin{eqnarray*}
\frac{ \partial G(\lambda,\sigma)}{\partial \lambda}&=& \frac{1}{\pi^2} \int_{0}^{2\pi} \int_{0}^{2\pi} \int_{0}^{+\infty} \lambda \cdot  r_1^5 /\tau^2 \cdot e^{-(1+\lambda^2)r_1^2/\tau^2} \cdot e^{2\sigma\lambda r_1^2 \cos(\theta_1-\theta_2)/\tau^2}  d r_1 d\theta_1d\theta_2\\
&=&  \frac{1}{\pi^2} \int_{0}^{2\pi} \int_{0}^{2\pi}  \frac{\lambda \tau^4}{(1+\lambda^2-2\lambda\sigma \cos(\theta_1-\theta_2))^3} d\theta_1d\theta_2\\
&=& 4\tau^4 \cdot \frac{\lambda(1+\lambda^4+ 2\lambda^2 +2\lambda^2\sigma^2)}{\sqrt{(1+\lambda^2+2\lambda\sigma)^5(1+\lambda^2-2\lambda\sigma)^5}}\\
&\le &  \frac{2\tau^4(\mu_{+}^2 +\mu_{-}^2) }{(\mu_{+} \mu_{-})^{5/2} },
\end{eqnarray*}
where $\mu_{+}:=1+\lambda^2+2\lambda\sigma \ge 0$ and $\mu_{-}:=1+\lambda^2-2\lambda\sigma \ge 0$. Let 
\[
f(\lambda,\sigma):=\frac{\tau^4(\mu_{+}^2 +\mu_{-}^2) }{(\mu_{+} \mu_{-})^{5/2} }.
\]
We next prove that $f(\lambda,\sigma)$ is a decreasing function with respect to $\sigma$ for any fixed $\lambda \le \sqrt{\frac{5-\sqrt{21}}{2}}$. In fact, through some basic algebraic  computation, we have
\begin{eqnarray*}
\frac{\partial f(\lambda,\sigma) }{\partial \sigma}&=&(1-\sigma^2) \cdot \frac{ \lambda(1-\sigma^2)(\mu_{+}-\mu_{-})(5\mu_{+}^2+4\mu_{+}\mu_{-}+5\mu_{-}^2) -4\sigma (\mu_{+}^2 +\mu_{-}^2)\mu_{+}\mu_{-}  }{(\mu_{+} \mu_{-})^{7/2}} \\
&\le &(1-\sigma^2) \sigma (\mu_{+}^2+\mu_{-}^2) \cdot \frac{ 28\lambda^2  (1-\sigma^2) -4 \mu_{+}\mu_{-}  }{(\mu_{+} \mu_{-})^{7/2}} \\
&= &(1-\sigma^2) \sigma (\mu_{+}^2+\mu_{-}^2) \cdot \frac{ 28\lambda^2 - 12\lambda^2\sigma^2 -4(1+\lambda^2)^2 }{(\mu_{+} \mu_{-})^{7/2}}\\
&\le &0,
\end{eqnarray*}
provided $\lambda \le \sqrt{\frac{5-\sqrt{21}}{2}}$.  Note that $G(0,\sigma)=0$. It then immediately gives
\[
G(\lambda,\sigma)\le 2\int_{0}^{\lambda} f(t,\sigma) dt \le 2\int_{0}^{\lambda} f(t,0) dt =4\int_{0}^{\lambda}  \frac{t}{(1+t^2)^3} dt =\frac{\lambda^2(\lambda^2+2)}{(\lambda^2+1)^2},
\]
which completes the claim \eqref{eq:exahsqu}.

\end{proof}

\begin{lemma} \label{le:hxratioxexpec}
Assume  $\lambda \ge 2.95$.  Let $\vx,\vh$ be two fixed vectors in $ \C^n$ with $\norm{\vx}=\norm{\vh}=1$ and  $\Im(\vh^* \vx)=0$. 
Suppose  $\xi \in \C^n$ is a complex Gaussian random vector.   Then  we have
\begin{equation*}
 \E \xkh{\frac{ \Re^2(\vh^* \xi \xi^* \vx )}{ |\xi^* \vx|^2} \cdot \1_{\dkh{\lambda |\xi^* \vx|\ge |\xi^* \vh|}}} \ge \frac 3{8} - \frac 1{(\lambda+1)^2}.
\end{equation*}
\end{lemma}
\begin{proof}
Due to the unitary invariance of Gaussian random vector, without loss of generality, we assume $\vh=\ve_1$ and $\vx=\sigma \ve_1 + \tau e^{i \phi} \ve_2$, where $\sigma=\vh^*\vx \in \R, \abs{\sigma} \le 1$ and $\tau=\sqrt{1-\sigma^2} $.
Let $\xi_1, \xi_2$ be the first and second entries of $\xi$. Denote $\xi_1=\xi_{1,\Re}+ i \xi_{1,\Im}$ and $\xi_2=\xi_{2,\Re}+ i \xi_{2,\Im}$ where $\xi_{1,\Re}, \xi_{1,\Im}, \xi_{2,\Re}, \xi_{2,\Im} $
are independent Gaussian random variables with distribution  $\mathcal{N}(0,1/2)$. Then the inequality $ \lambda |\xi^* \vx| \ge |\xi^* \vh|$ is equivalent to 
\begin{eqnarray*}
\lambda \sqrt{\xkh{\sigma\xi_{1,\Re} +\tau (\cos\phi \xi_{2,\Re}+\sin\phi \xi_{2,\Im}) }^2+\xkh{\sigma\xi_{1,\Im} -\tau (\sin\phi \xi_{2,\Re}-\cos\phi \xi_{2,\Im}) }^2 } \ge \sqrt{(\xi_{1,\Re}^2+\xi_{1,\Im}^2 )}.
\end{eqnarray*}
To obtain the conclusion, we take the polar coordinates transformations and denote
\[
\left.
\begin{array}{r}
\xi_{1,\Re}= r_1 \cos \theta_1 \\
\xi_{1,\Im}= r_1 \sin \theta_1 \\
\sigma\xi_{1,\Re} +\tau (\cos\phi \xi_{2,\Re}+\sin\phi \xi_{2,\Im}) =r_2 \cos \theta_2\\
\sigma\xi_{1,\Im} -\tau (\sin\phi \xi_{2,\Re}-\cos\phi \xi_{2,\Im})= r_2 \sin \theta_2
\end{array} \right\}
\]
with $r_1, r_2 \in (0,+\infty), \theta_1,\theta_2 \in [0,2\pi]$.  It is easy to check that 
\begin{eqnarray*}
 \Re (\vh^* \xi \xi^* \vx ) &=& \Re\xkh{(\sigma \xi_1 + \tau e^{-i \phi} \xi_2) \bar{\xi}}\\
 &=& \xi_{1,\Re} \xkh{\sigma\xi_{1,\Re} +\tau (\cos\phi \xi_{2,\Re}+\sin\phi \xi_{2,\Im}) }+\xi_{1,\Im}\xkh{\sigma\xi_{1,\Im} -\tau (\sin\phi \xi_{2,\Re}-\cos\phi \xi_{2,\Im})}\\
 &=& r_1r_2 \cos(\theta_1-\theta_2).
\end{eqnarray*}
It means the expectation can be written as
\begin{eqnarray*}
F(\lambda,\sigma)&:=& \E \xkh{\frac{ \Re^2(\vh^* \xi \xi^* \vx )}{|\xi^* \vx|^2} \cdot \1_{\dkh{\lambda |\xi^* \vx|\ge |\xi^* \vh|}}}\\
&=&\frac{1}{\pi^2} \int_{0}^{2\pi} \int_{0}^{2\pi} \int_{0}^{+\infty} \int_{0}^{\lambda\cdot r_2} \frac{r_1^3 r_2}{\tau^2} \cdot  \cos^2(\theta_1-\theta_2) \cdot e^{-(r_1^2+r_2^2)/\tau^2} \cdot e^{2\sigma r_1r_2 \cos(\theta_1-\theta_2)/\tau^2} d r_1 d r_2 d\theta_1d\theta_2\\
&=& 2\sum_{k=0}^{\infty} \frac{2k+1}{(k!)^2\cdot (k+1)} \cdot \frac{\sigma^{2k}}{\tau^{4k+2}} \int_{0}^{+\infty} \int_{0}^{\lambda\cdot r_2} r_1^{2k+3} r_2^{2k+1} \cdot e^{-(r_1^2+r_2^2)/\tau^2}  d r_1 d r_2,
\end{eqnarray*}
where the last equation follows from the fact
\[
\int_{0}^{2\pi} \int_{0}^{2\pi}  \cos^{2k}(\theta_1-\theta_2) d\theta_1d\theta_2=\frac{(2k-1)!!}{2^{k-2}\cdot k!}
\]
for any integer $k$. To evaluate $F(\lambda,\sigma)$, we first take the derivative and then obtain
\begin{eqnarray*}
\frac{ \partial F(\lambda,\sigma)}{\partial \lambda}&:=&2\sum_{k=0}^{\infty} \frac{2k+1}{(k!)^2\cdot (k+1)} \cdot \frac{\sigma^{2k}}{\tau^{4k+2}} \int_{0}^{+\infty}  \lambda^{2k+3} r_2^{4k+5} \cdot e^{-(1+\lambda^2)r_2^2/\tau^2}   d r_2 \\
&= &  2\sum_{k=0}^{\infty} \frac{(2k+1)!(2k+1)}{(k!)^2} \cdot \sigma^{2k}(1-\sigma^2)^2 \cdot \xkh{\frac{\lambda}{1+\lambda^2}}^{2k+3}.
\end{eqnarray*}
Since $F(0,\sigma)=0$, it implies that
\begin{equation} \label{eq:Fvalue}
F(\lambda,\sigma)=2\sum_{k=0}^{\infty} \frac{(2k+1)!(2k+1)}{(k!)^2} \cdot \sigma^{2k}(1-\sigma^2)^2 \cdot \int_0^{\lambda} \xkh{\frac{t}{1+t^2}}^{2k+3} dt.
\end{equation}
With this in place, all we need to do is to lower bound the integral $\int_0^{\lambda} \xkh{\frac{t}{1+t^2}}^{2k+3} dt$. Note that 
\[
\int_0^{\lambda} \xkh{\frac{t}{1+t^2}}^{2k+3} dt =\int_0^{1} \xkh{\frac{t}{1+t^2}}^{2k+3} dt+\int_1^{\lambda} \xkh{\frac{t}{1+t^2}}^{2k+3} dt := \RNum{1}+\RNum{2}.
\]
For the first term, let $t=\tan\theta$. It then gives
\begin{eqnarray}
 \RNum{1} &=& \int_0^{\frac \pi4} \sin^{2k+3}\theta \cos^{2k+1}\theta~ d\theta \nonumber\\
 &=& \frac{1}{2^{2k+2}} \int_0^{\frac \pi4}  \sin^{2k+1}(2\theta)(1-\cos(2\theta))~d\theta \nonumber \\
  &=&\frac{k!}{(2k+1)!!\cdot 2^{k+3}}-\frac 1{2(k+1)\cdot 2^{2k+3}}. \label{eq:firstterm}
\end{eqnarray}
For the second term, noting that $\lambda\ge 1$, we have
\begin{equation} \label{eq:secterm}
 \RNum{2} \ge  \int_1^{\lambda} (1+t)^{-2k-3} dt=\frac 1{2(k+1)} \cdot \xkh{\frac 1{2^{2k+2}}-\frac1{(\lambda+1)^{2k+2}}}.
\end{equation}
Putting  \eqref{eq:firstterm} and \eqref{eq:secterm}   into \eqref{eq:Fvalue}, we have 
\begin{eqnarray*}
F(\lambda,\sigma)&\ge&\sum_{k=0}^{\infty} \frac{(2k+1)!! (2k+1)}{(k+1)!} \cdot \sigma^{2k}(1-\sigma^2)^2 \cdot \xkh{\frac{(k+1)!}{4(2k+1)!!}+\frac 1{2^{k+3}}-\frac{2^{k}}{(1+\lambda)^{2k+2}}}.
\end{eqnarray*}
Let $\beta:=(1+\lambda)^2$. Expand $F(\lambda,\sigma)$ into a series with respect to $\sigma$ and we have
\begin{equation} \label{eq:Fseriesform}
F(\lambda,\sigma)\ge  \frac 38 -\frac 1{\beta}+\frac 9{16} \sigma^2- \sum_{k=1}^{\infty} \frac{(2k-1)!!(2k+7)}{2^{k+4}(k+2)!}\sigma^{2k+2}-(1-\frac 4{\beta})^2\cdot \sum_{k=1}^{\infty} \frac{2^k (2k-1)!!}{(k-1)! \beta^k}\sigma^{2k+2},
\end{equation}
where we use the fact that  $\lambda \ge 2.95$ in the above inequality. Next, we need to upper bound the last two series. From Wallis' inequality \cite{kazarinoff}, we know
\[
\frac{(2k-1)!!}{2^k k!}\le \frac 1{\sqrt{2k}}.
\]
Thus 
\begin{eqnarray}
 \sum_{k=1}^{\infty} \frac{(2k-1)!!(2k+7)}{2^{k+4}(k+2)!}\sigma^{2k+2} & \le &  \frac {\sigma^{2}}{16} \cdot \sum_{k=1}^{\infty} \frac{2k+7}{ \sqrt{2k}(k+2)(k+1)} \nonumber\\
 &\le &  \frac {\sigma^{2}}{8}\cdot \sum_{k=1}^{\infty} \frac{1}{k^{3/2}}\nonumber\\
 &\le &  \frac {\sigma^{2}}{8} \xkh{1+\int_1^{\infty} \frac{1}{t^{3/2}}~dt}\nonumber\\
 &=&  \frac {3}{8}\sigma^{2}, \label{eq:firserei1}
\end{eqnarray}
where the second inequality follows from the fact that $\frac{2k+7}{(k+2)(k+1)} \le \frac{2\sqrt{2}}{k}$ for all $k\ge 1$.

On the other hand, using Wallis' inequality again, we have
\begin{eqnarray}
(1-\frac 4{\beta})^2\cdot \sum_{k=1}^{\infty} \frac{2^k (2k-1)!!}{(k-1)! \beta^k}\sigma^{2k+2} & \le & (1-\frac 4{\beta})^2\cdot \sigma^2  \cdot \sum_{k=1}^{\infty} \frac{4^k\cdot  k}{\sqrt{2k} \beta^k} \nonumber\\
 &\le &  (1-\frac 4{\beta})^2\cdot \frac{\sigma^2}{\sqrt{2}}  \cdot \sum_{k=1}^{\infty} k\xkh{\frac{4}{\beta}}^k\nonumber \\
 &= &  \frac {2\sqrt{2}}{\beta} \sigma^{2}, \label{eq:secseries2}
\end{eqnarray}
where the last equation follows from the fact that 
\[
\sum_{k=1}^{\infty}  k x^{k-1} =\frac{1}{(1-x)^2}  \quad \mbox{for all} \quad 0\le x<1.
\]
Putting \eqref{eq:firserei1} and \eqref{eq:secseries2} into \eqref{eq:Fseriesform}, we know that for $\lambda\ge 2.95$, it holds
\[
F(\lambda,\sigma)\ge  \frac 38 -\frac 1{\beta}.
\]
This completes the proof.

\end{proof}

\end{document}